\documentclass[12pt]{amsart}
\usepackage[dvips,centering,includehead,width=15.1cm,height=22.7cm]{geometry}

\usepackage{amsfonts, amsmath, amssymb, amsthm, pst-node, epsfig, enumerate}
\usepackage{pstricks}

\theoremstyle{plain}
\newtheorem{thm}{Theorem}[section]
\newtheorem{prop}[thm]{Proposition}
\newtheorem{cor}[thm]{Corollary}

\newtheorem{lemma}[thm]{Lemma}

\theoremstyle{definition}
\newtheorem{defn}[thm]{Definition}
\newtheorem{example}[thm]{Example}
\newtheorem{remark}[thm]{Remark}

\numberwithin{equation}{section}

\def\wt{\operatorname{wt}}
\def\xing{\operatorname{xing}}
\def\loop{\operatorname{loop}}
\def\wind{\operatorname{wind}}

\def\P{\mathbf{P}}

\begin{document}

\title[Pl\"ucker coordinates associated with a planar network]{A
  formula for Pl\"ucker coordinates\\[.05in] associated with a planar network}
\author{Kelli Talaska}
\address{Department of Mathematics, University of Michigan,
Ann Arbor, MI 48109, USA} \email{kellicar@umich.edu}
\thanks{The author was partially supported by NSF Grants
CCF-04.30201 and DMS-0502170.}

\date{April 14, 2008 
}



\begin{abstract}
For a planar directed graph~$G$, Postnikov's boundary measurement
map sends positive weight functions on the edges of~$G$ onto the
appropriate totally nonnegative Grassmann cell. We establish an
explicit formula for Postnikov's map by expressing each Pl\"ucker
coordinate as a ratio of two combinatorially defined polynomials
in the edge weights, with positive integer coefficients.  In the
non-planar setting, we show that a similar formula holds for
special choices of Pl\"ucker coordinates.
\end{abstract}

\maketitle

\section{Introduction}

Totally nonnegative Grassmannians are an important subclass of general
totally nonnegative homogeneous spaces, first introduced and
studied by G.~Lusztig and K.~Rietsch (see, e.g.,~\cite{lusztig,
rietsch, rietsch2}). Informally speaking, the totally nonnegative
Grassmannian is the part of a real Grassmann manifold where all
Pl\"ucker coordinates are nonnegative. A.~Postnikov's
groundbreaking paper~\cite{post} established combinatorial
foundations for the study of totally nonnegative Grassmannians, in
particular providing the tools required for the construction of
cluster algebra structures in (ordinary) Grassmannians by
J.~Scott~\cite{scott}, and for the exploration of tropical
analogues by D.~Speyer and L.~Williams~\cite{speyer-williams}.

The goal of this paper is to give an explicit combinatorial
formula describing the main construction in~\cite{post}: the
\emph{boundary measurement map} that assigns a point in the
totally nonnegative Grassmannian to each planar directed network
with positive edge weights. To state our main results, we will
need to quickly recall the main features of Postnikov's
construction;  for complete details, see
Section~\ref{sec:boundary-meas}.

The construction begins with a planar directed graph~$G$ properly
embedded in a disk. Every vertex of $G$ lying on the boundary of
the disk is assumed to be a source or a sink. Each edge of~$G$ is
assigned a \emph{weight}, which we treat as a formal variable.
Postnikov defines the \emph{boundary measurement matrix}~$A$ with
columns labeled by the boundary vertices and rows labeled by the
set $I$ of boundary sources, as follows. Each matrix entry of~$A$
is, up to a sign that accounts for how the sources and sinks
interlace along the boundary, a weight generating function for
directed walks from a given boundary source to a given boundary
vertex, where each walk is counted with a sign reflecting the
parity of its topological winding index. The maximal minors
$\Delta_J(A)$ of the boundary measurement matrix~$A$ (here $J$ is
a subset of boundary vertices with $|J|=|I|$) are then interpreted
as Pl\"ucker coordinates of a point in a Grassmannian. The fact
that these minors are nonnegative (so that we get a point in a
totally nonnegative Grassmannian) follows from the assertion in
\cite{post} that each maximal minor $\Delta_J(A)$ can be written
as a \emph{subtraction-free} rational expression in the edge
weights.

Postnikov's proof of this fact is recursive.  In this paper, we provide
a direct proof via an explicit combinatorial formula for the minors
$\Delta_J(A)$, writing each of them as a ratio of two polynomials in the
edge weights, with positive integer coefficients.

Here we state the main result in the case that $G$ is
\emph{perfectly oriented}, i.e., every interior vertex of $G$ has
exactly one incoming edge or exactly one outgoing edge (or both).
We will later address an extension to graphs which are planar but
not necessarily perfectly oriented and discuss an analogue for
non-planar perfectly oriented graphs.

In order to state our formula, we will need the following notions.
A \emph{conservative flow} in a perfectly oriented graph~$G$ is a
(possibly empty) collection of pairwise vertex-disjoint oriented
cycles. (Each cycle is self-avoiding, i.e., it is not allowed to
pass through a vertex more than once. For perfectly oriented
graphs~$G$, this is equivalent to not repeating an edge.) For
$|J|=|I|$, a \emph{flow from $I$ to $J$} is a collection of
self-avoiding walks and cycles, all pairwise vertex-disjoint, such
that each walk connects a source in $I$ to a boundary vertex in
$J$.  (A vertex may be connected to itself by a walk with no
edges.) The \emph{weight} of a flow (conservative or not) is the
product of the weights of all its edges. A flow with no edges has
weight~1.

\begin{thm}
\label{th:main-intro}
The Pl\"ucker coordinate $\Delta_J(A)$
is given by $\Delta_J=\frac{f}{g}$, where $f$ and $g$ are nonnegative
polynomials in the edge weights, defined as follows:
\begin{itemize}
\item $f$ is the weight generating function for all flows from $I$ to~$J$;
\item $g$ is the weight generating function for all conservative
  flows in~$G$.
\end{itemize}
\end{thm}

If the underlying graph $G$ is acyclic, then $g=1$, and
Theorem~\ref{th:main-intro} reduces to the well known result of
B.~Lindstr\"om~\cite{lind} expressing the determinant of a matrix
associated with a planar acyclic network in terms of
non-intersecting paths; see, e.g., \cite{tptp} and references
therein. Thus, Theorem~\ref{th:main-intro} can be viewed as a
generalization of Lindstr\"om's result to non-acyclic planar
networks. Another such generalization was given by
S.~Fomin~\cite{lew} whose setup differed from Postnikov's in that
the analogues of boundary measurements did not involve any signs.
In Fomin's approach, total nonnegativity is achieved---for edge
weights specialized to nonnegative real values---by writing the
minors in question as formal power series with nonnegative
coefficients. In contrast, Postnikov's map produces
subtraction-free expressions which are rational (that is, involve
division) but finite (that is, do not require infinite summation).

The rest of the paper is organized as follows. In
Section~\ref{sec:boundary-meas}, we review Postnikov's
construction of boundary measurements in planar circular networks.
In Section~\ref{sec:involution}, we present our main result in the
perfectly oriented case (Theorem~\ref{th:mainthm}, a more formal
restatement of Theorem \ref{th:main-intro}), and provide a proof
based on a sign-reversing involution.
We then extend these results in two directions.
 In Section~\ref{sec:general}, we give a generalization
 of Theorem~\ref{th:mainthm}, extending to planar networks which are
not necessarily perfectly oriented.
In Section~\ref{sec:non-planar} we examine non-planar networks
with perfect orientations, using G.~Lawler's notion of
loop-erasure in place of the topological winding index. In this
generality, we show that the formula in Theorem~\ref{th:mainthm}
holds for those Pl\"ucker coordinates which are equal to
individual entries in the boundary measurement matrix, but not for
arbitrary minors $\Delta_J$.

\section{Boundary measurements in perfectly oriented networks}
\label{sec:boundary-meas}

\begin{defn}\label{def:circ}
A \emph{planar circular directed graph} is a finite directed graph~$G$
properly embedded in a closed oriented disk
(so that its edges intersect only at the appropriate vertices),
together with a distinguished labeled subset
$\{b_1\ldots,b_n\}$ of \emph{boundary vertices} such that
\begin{enumerate}
\item
$b_1, \ldots, b_n$ appear in clockwise
order around the boundary of the disk,
\item
all other vertices of~$G$ lie
in the interior of the disk, and
\item
each boundary vertex $b_i$ is incident to at most one edge.
\end{enumerate}
A non-boundary vertex in~$G$ is called an \emph{interior vertex}.
Loops and multiple edges are permitted. Each boundary vertex is
designated a source or a sink, even if it is an isolated vertex.
We denote by $I\subseteq [n]=\{1,\dots,n\}$ the indexing set for the boundary
sources of~$G$, so that these sources form the set $\{b_i:i\in
I\}$.

A~\emph{planar circular network} $N=(G,x)$ is a planar circular
directed graph $G$ together with a collection $x=(x_e)$ of formal
variables~$x_e$ labeled by the edges $e$ in~$G$. We call $x_e$ the
\emph{weight} of~$e$.
\end{defn}

\begin{defn}[\cite{post}]\label{def:perf}
A planar circular directed graph (or network) is said to be
  \emph{perfectly oriented} if every interior vertex either has exactly one outgoing
  edge (with all other edges incoming) or exactly one incoming
  edge (with all other edges outgoing).
\end{defn}

For example, let $G$ be a circular directed graph in which all
interior vertices are trivalent, with no interior sources or
sinks. Then $G$ is perfectly oriented. Such a graph is shown in
Figure~\ref{fig:mainex}; this will serve as our running example
throughout Sections~\ref{sec:boundary-meas}
and~\ref{sec:involution}.

\begin{figure}[ht]
\psset{unit=10pt}
\begin{pspicture}(-2,-2)(37,27)

\psline[linestyle=dotted](0,0)(0,25)(25,25)(25,0)(0,0)

\psline[linewidth=2pt,arrows=*->,arrowsize={1.5pt 4}](5,5)(5,2)
\psline[linewidth=2pt,arrows=-*,arrowsize={1.5pt 4}](5,3)(5,0)
\psline[linewidth=2pt,arrows=->,arrowsize={1.5pt 4}](5,5)(5,8)
\psline[linewidth=2pt,arrows=-*,arrowsize={1.5pt 4}](5,7)(5,10)
\psline[linewidth=2pt,arrows=*->,arrowsize={1.5pt 4}](5,15)(5,12)
\psline[linewidth=2pt,arrows=-*,arrowsize={1.5pt 4}](5,13)(5,10)
\psline[linewidth=2pt,arrows=->,arrowsize={1.5pt 4}](5,15)(5,18)
\psline[linewidth=2pt,arrows=-*,arrowsize={1.5pt 4}](5,17)(5,20)
\psline[linewidth=2pt,arrows=*->,arrowsize={1.5pt 4}](0,20)(3,20)
\psline[linewidth=2pt,arrows=-,arrowsize={1.5pt 4}](2,20)(5,20)
\psline[linewidth=2pt,arrows=*->,arrowsize={1.5pt 4}](5,20)(8,20)
\psline[linewidth=2pt,arrows=-,arrowsize={1.5pt 4}](7,20)(10,20)
\psline[linewidth=2pt,arrows=*->,arrowsize={1.5pt
4}](10,20)(10,23) \psline[linewidth=2pt,arrows=-*,arrowsize={1.5pt
4}](10,22)(10,25)

\psline[linewidth=2pt,arrows=->,arrowsize={1.5pt 4}](10,20)(10,17)
\psline[linewidth=2pt,arrows=-*,arrowsize={1.5pt 4}](10,18)(10,15)
\psline[linewidth=2pt,arrows=->,arrowsize={1.5pt 4}](10,15)(7,15)
\psline[linewidth=2pt,arrows=-,arrowsize={1.5pt 4}](8,15)(5,15)
\psline[linewidth=2pt,arrows=->,arrowsize={1.5pt 4}](5,10)(8,10)
\psline[linewidth=2pt,arrows=-*,arrowsize={1.5pt 4}](7,10)(10,10)

\psline[linewidth=2pt,arrows=->,arrowsize={1.5pt 4}](10,10)(10,7)
\psline[linewidth=2pt,arrows=-*,arrowsize={1.5pt 4}](10,8)(10,5)
\psline[linewidth=2pt,arrows=->,arrowsize={1.5pt 4}](10,5)(7,5)
\psline[linewidth=2pt,arrows=-,arrowsize={1.5pt 4}](8,5)(5,5)

\psline[linewidth=2pt,arrows=*->,arrowsize={1.5pt 4}](15,5)(12,5)
\psline[linewidth=2pt,arrows=-,arrowsize={1.5pt 4}](13,5)(10,5)
\psline[linewidth=2pt,arrows=->,arrowsize={1.5pt 4}](10,10)(13,10)
\psline[linewidth=2pt,arrows=-*,arrowsize={1.5pt 4}](12,10)(15,10)
\psline[linewidth=2pt,arrows=->,arrowsize={1.5pt 4}](15,5)(15,8)
\psline[linewidth=2pt,arrows=-,arrowsize={1.5pt 4}](15,7)(15,10)

\psline[linewidth=2pt,arrows=->,arrowsize={1.5pt 4}](15,10)(18,10)
\psline[linewidth=2pt,arrows=-*,arrowsize={1.5pt 4}](17,10)(20,10)
\psline[linewidth=2pt,arrows=->,arrowsize={1.5pt 4}](20,10)(20,13)
\psline[linewidth=2pt,arrows=-*,arrowsize={1.5pt 4}](20,12)(20,15)

\psline[linewidth=2pt,arrows=->,arrowsize={1.5pt 4}](20,10)(20,7)
\psline[linewidth=2pt,arrows=-*,arrowsize={1.5pt 4}](20,8)(20,5)
\psline[linewidth=2pt,arrows=->,arrowsize={1.5pt 4}](20,5)(17,5)
\psline[linewidth=2pt,arrows=-,arrowsize={1.5pt 4}](18,5)(15,5)
\psline[linewidth=2pt,arrows=*->,arrowsize={1.5pt 4}](25,5)(22,5)
\psline[linewidth=2pt,arrows=-,arrowsize={1.5pt 4}](23,5)(20,5)

\psline[linewidth=2pt,arrows=->,arrowsize={1.5pt
4}](10,15)(15.5,15)
\psline[linewidth=2pt,arrows=-,arrowsize={1.5pt 4}](15,15)(20,15)
\psline[linewidth=2pt,arrows=->,arrowsize={1.5pt
4}](20,15)(20,20.5)
\psline[linewidth=2pt,arrows=-*,arrowsize={1.5pt 4}](20,20)(20,25)

\uput[r](25,5){\large{$b_4$}} \uput[u](20,25){\large{$b_3$}}
\uput[u](10,25){\large{$b_2$}} \uput[l](0,20){\large{$b_1$}}
\uput[d](5,0){\large{$b_5$}}

\uput[r](20,7.5){$w_1$} \uput[r](15,7.5){$w_3$}
\uput[u](17.5,5){$w_2$} \uput[u](17.5,10){$w_4$}
\uput[r](10,7.5){$y_1$} \uput[r](5,7.5){$y_3$}
\uput[u](7.5,5){$y_2$} \uput[u](7.5,10){$y_4$}
\uput[r](10,17.5){$z_1$} \uput[r](5,17.5){$z_3$}
\uput[u](7.5,15){$z_2$} \uput[u](7.5,20){$z_4$}

\uput[u](22.5,5){$a_4$} \uput[r](20,20){$a_3$}
\uput[r](10,22.5){$a_2$} \uput[u](2.5,20){$a_1$}
\uput[r](5,2.5){$a_5$} \uput[r](20,12.5){$c$}
\uput[r](5,12.5){$d$} \uput[u](12.5,5){$f$} \uput[u](12.5,10){$g$}
\uput[u](15,15){$h$}

\uput[r](27,21){The cycles of $N$} \uput[r](27,19){have weights}
\uput[r](27,17){$W=w_1w_2w_3w_4$,}
\uput[r](27,15){$Y=y_1y_2y_3y_4$,}
\uput[r](27,13){$Z=z_1z_2z_3z_4$, and}
\uput[r](27,11){$T=fy_2y_3y_4gw_4w_1w_2$.}

\end{pspicture}
\caption{A perfectly oriented planar circular network $N$
with boundary vertices $b_1, b_2, b_3, b_4, b_5$.  Edges are
labeled by their weights.}
\label{fig:mainex}
\end{figure}

A \emph{walk} $P=(e_1, \ldots, e_m)$ in~$G$ is formed by
traversing the edges $e_1, e_2, \ldots, e_m$ in the specified
order.  (The head of $e_i$ is the tail of~$e_{i+1}$.) We write
$P:u\leadsto v$ to indicate that $P$ is a walk starting at a
vertex $u$ and ending at a vertex~$v$. Note that in a perfectly
oriented circular graph, any self-intersecting walk between
boundary vertices must repeat at least one edge at every point of
self-intersection.

Define the \emph{weight} of a walk $P=(e_1, \ldots, e_m)$ to be
\[
\wt(P)= x_{e_1}\cdots x_{e_m} .
\]

A walk $P:u\leadsto u$ with no edges is called a \emph{trivial
walk} and has weight 1.

\begin{defn}[\cite{post}]\label{def:wind}
Let $P:u\leadsto v$
be a non-trivial walk in a planar circular directed graph~$G$
connecting boundary
vertices $u$ and~$v$. Performing an isotopy if necessary,
we may assume that the tangent vector to $P$ at $u$ has the same
direction as the tangent vector to $P$ at~$v$.  The \emph{winding
index} $\wind(P)$ is the signed number of full $360^\circ$ turns
the tangent vector makes as we travel along~$P$, counting
counterclockwise turns as positive. For a trivial walk~$P$, we set
$\wind(P)=1$.
\end{defn}

\begin{defn} [\cite{post}]
\label{def:bdry-meas} For boundary vertices $b_i$ and $b_j$ in a
planar circular network~$N$, the \emph{boundary measurement}
$M_{ij}$ is the formal power series
\begin{equation}
\label{eq:Mij} M_{ij}=\sum_{P:b_i\leadsto
b_j}(-1)^{\wind(P)}\wt(P),
\end{equation}
the sum over all directed walks $P:b_i \leadsto b_j$.
\end{defn}

\begin{example}
In the circular network $N$ shown in Figure~\ref{fig:mainex},
any walk $P$ from $b_1$ to $b_2$
consists of the edges with weights $a_1,z_4,a_2$, together with
some number of repetitions of the cycle of weight
$Z=z_4z_1z_2z_3$. Consequently,
\[
M_{12}=a_1z_4a_2-a_1Zz_4a_2+a_1Z^2z_4a_2-a_1Z^3z_4a_2+\ldots=\frac{a_1z_4a_2}{1+Z}\,.
\]
\end{example}

\begin{defn}\label{def:boundarymeasmatrix}
Let $N$ be a planar circular network. Let $I=\{i_1<\cdots<i_k\}$,
so that the boundary sources, listed clockwise, are
$b_{i_1},b_{i_2},\dots,b_{i_k}$. The \emph{boundary measurement
matrix} $A(N)=(a_{tj})$ is the $k\times n$ matrix defined by
\[
a_{tj} = (-1)^{s(i_t,j)} M_{i_t,j}\,,
\]
where $s(i_t,j)$ denotes the number of elements of $I$ strictly between
$i_t$ and $j$.
\end{defn}

Let $\Delta_J(A(N))$ denote the $k\times k$ minor of $A(N)$ whose
columns are indexed by~$J$.  That is,
$\Delta_J(A(N))=\det(a_{tj})_{t\in [1,k], j\in J}$. When no
confusion will arise, we may simply write $\Delta_J$. We note that
each nontrivial boundary measurement $M_{ij}$ (i.e. when $i$ is a
source and $j$ is a sink) occurs as the minor
$\Delta_{I\setminus\{i\}\cup\{j\}}$.

\begin{example}
Suppose that $N$ is the planar circular network in
Figure~\ref{fig:mainex}. Then the boundary source set is indexed
by $I=\{1,4\}$, and we have
\[A(N)=\left(%
\begin{array}{ccccc}
  1 & M_{12} & M_{13} & 0 & -M_{15} \\
  0 & M_{42} & M_{43} & 1 & M_{45} \\
\end{array}%
\right).\]

\end{example}

\begin{thm}[\cite{post}]
\label{th:sfre} If $N=(G,x)$ is a planar circular network, then
each maximal minor $\Delta_J$ of the boundary measurement matrix
can be written as a subtraction-free rational expression in the
edge weights $x_e$.
\end{thm}

Postnikov's proof of Theorem~\ref{th:sfre} is inductive. In
Sections~\ref{sec:involution} and~\ref{sec:general}, we will give
an explicit combinatorial formula for the boundary measurements in
a planar circular network, providing a constructive proof.
Theorem~\ref{th:mainthm} gives the formula in the perfectly
oriented case, and Corollary~\ref{th:main-non-perfect} generalizes
this result, giving a formula for an arbitrary planar circular
network.

\begin{defn} [\cite{post}]
\label{def:xing-number} Let $\pi:I\to J$ be a bijection such that
$\pi(i)=i$ for all $i\in I\cap J$. A pair of indices
$(i_1,i_2)$, where $\{i_1<i_2\}\subset I\setminus J$, is called a
\emph{crossing}, an \emph{alignment}, or a \emph{misalignment}
of~$\pi$, if the two directed chords $[b_{i_1},b_{\pi(i_1)}]$ and
$[b_{i_2},b_{\pi(i_2)}]$ are arranged with respect to each other
as shown in Figure~\ref{fig:crossing_alignment_misalignment}.
Define the {\it crossing number\/} $\xing(\pi)$ of $\pi$ as the
number of crossings of~$\pi$.

\psset{unit=1pt}
\begin{figure}[ht]
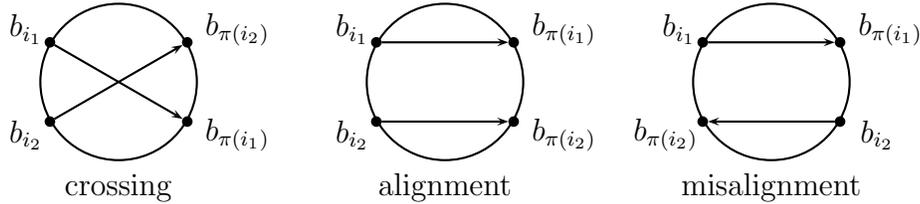

\pspicture(-50,-40)(50,40) \rput(0,-40){crossing}
\pscircle[linecolor=black](0,0){30}
\cnode*[linewidth=0,linecolor=black](25.98,15){2}{B1}
\cnode*[linewidth=0,linecolor=black](25.98,-15){2}{B2}
\cnode*[linewidth=0,linecolor=black](-25.98,-15){2}{B4}
\cnode*[linewidth=0,linecolor=black](-25.98,15){2}{B5}
\ncline{->}{B5}{B2} \ncline{->}{B4}{B1} \rput(-35,20){$b_{i_1}$}
\rput(45,20){$b_{\pi(i_2)}$} \rput(-35,-20){$b_{i_2}$}
\rput(45,-20){$b_{\pi(i_1)}$}
\endpspicture
\qquad
\pspicture(-50,-40)(50,40) \pscircle[linecolor=black](0,0){30}
\rput(0,-40){alignment}
\cnode*[linewidth=0,linecolor=black](25.98,15){2}{B1}
\cnode*[linewidth=0,linecolor=black](25.98,-15){2}{B2}
\cnode*[linewidth=0,linecolor=black](-25.98,-15){2}{B4}
\cnode*[linewidth=0,linecolor=black](-25.98,15){2}{B5}
\ncline{->}{B5}{B1} \ncline{->}{B4}{B2} \rput(-35,20){$b_{i_1}$}
\rput(45,20){$b_{\pi(i_1)}$} \rput(-35,-20){$b_{i_2}$}
\rput(45,-20){$b_{\pi(i_2)}$}
\endpspicture
\qquad
\pspicture(-50,-40)(50,40) \pscircle[linecolor=black](0,0){30}
\rput(0,-40){misalignment}
\cnode*[linewidth=0,linecolor=black](25.98,15){2}{B1}
\cnode*[linewidth=0,linecolor=black](25.98,-15){2}{B2}
\cnode*[linewidth=0,linecolor=black](-25.98,-15){2}{B4}
\cnode*[linewidth=0,linecolor=black](-25.98,15){2}{B5}
\ncline{->}{B5}{B1} \ncline{<-}{B4}{B2} \rput(-35,20){$b_{i_1}$}
\rput(45,20){$b_{\pi(i_1)}$} \rput(-40,-20){$b_{\pi(i_2)}$}
\rput(40,-20){$b_{i_2}$}
\endpspicture
\caption{Crossings, alignments, and misalignments}
\label{fig:crossing_alignment_misalignment}
\end{figure}
\psset{unit=1pt}
\end{defn}

\begin{lemma}\label{alg-xing}
For distinct $i_1, i_2, j_1, j_2\in [n]$,
the chords $[b_{i_1},b_{j_1}]$ and $[b_{i_2},b_{j_2}]$ cross if
and only if \[(i_1-j_2)(j_2-j_1)(j_1-i_2)(i_2-i_1)<0.\]
\end{lemma}

\begin{proof}
This is a simple verification, left to the reader.
\end{proof}

Lemma~\ref{alg-xing} immediately leads to the following corollary.

\begin{cor}\label{misalign}
Let $\pi:I\to J$ be a bijection such that $\pi(i)=i$ for all $i\in
I\cap J$. For $\{i_1<i_2\}\subset I\setminus J$, the following are
equivalent:
\begin{enumerate}
\item $(i_1, i_2)$ is a misalignment; \item the chords
$[b_{i_1},b_{i_2}]$ and $[b_{\pi(i_1)},b_{\pi(i_2)}]$ cross; \item
$(i_1-\pi(i_2))(\pi(i_2)-i_2))(i_2-\pi(i_1))(\pi(i_1)-i_1)<0$.
\end{enumerate}
\end{cor}

We provide a new proof of the following result.

\begin{prop}[\cite{post}]
\label{xing} Let $I$ index the boundary sources of a planar
circular network~$N$ and let $J\subseteq [n]$, with $|J|=|I|$.
Then
\begin{equation}
\label{eq:delta(A(N))}
\Delta_J(A(N))=\sum_{\pi:I\rightarrow
J}(-1)^{\xing(\pi)}\prod_{i\in I}M_{i,\pi(i)},
\end{equation}
the sum over all bijections $\pi$ from $I$ to $J$.
\end{prop}

\begin{proof}
Taking the appropriate determinant, we see that
\[\Delta_J(A(N))=\sum_{\pi:I\rightarrow
J}(-1)^{\operatorname{inv}(\pi)}\prod_{i\in
I}(-1)^{s(i,\pi(i))}M_{i,\pi(i)}, \] where $s(i,\pi(i))$ is
defined as in Definition~\ref{def:boundarymeasmatrix} and
$\operatorname{inv}(\pi)$ is the number of inversions of $\pi$.
Here, an \emph{inversion} of $\pi$ is a pair $(i_1,i_2)$ with
$i_1<i_2$ and $\pi(i_1)>\pi(i_2)$.  Note that $\prod_{i\in
I}M_{i,\pi(i)}=0$ unless $\pi(i)=i$ for all $i\in I\cap J$. Thus,
we wish to show that if $\pi$ fixes the elements in $I\cap J$,
then
\begin{equation}\label{eq:inv-xing}(-1)^{\xing(\pi)}=(-1)^{\operatorname{inv}(\pi)}\prod_{i\in
I}(-1)^{s(i,\pi(i))}.\end{equation}
Consider the right-hand side of (\ref{eq:inv-xing}).
Each pair $(i_1,i_2)$ with $i_1<i_2$
contributes a factor of $\operatorname{sgn}((\pi(i_2)-\pi(i_1))$ to
$(-1)^{\operatorname{inv}(\pi)}$. Furthermore, $i_1$
contributes a factor of
$\operatorname{sgn}((i_1-i_2)(i_1-\pi(i_2))
=-\operatorname{sgn}(i_1-\pi(i_2))$
to $(-1)^{s(i_2,\pi(i_2))}$, since this product is negative if and
only if $\pi(i_2)<i_1<i_2$.  Similarly, $i_2$ contributes a factor
of
$\operatorname{sgn}((i_2-i_1)(i_2-\pi(i_1)))
=-\operatorname{sgn}((i_2-i_1)(\pi(i_1)-i_2))$
to $(-1)^{s(i_1,\pi(i_1))}$. Thus, the total contribution by the
pair $(i_1,i_2)$ is
\begin{equation*}
\operatorname{sgn}[(i_2-i_1)(i_1-\pi(i_2))(\pi(i_2)-\pi(i_1))(\pi(i_1)-i_2)].
\end{equation*}
Taking the product over all pairs $\{i_1<i_2\}$, we
get $(-1)^{\xing(\pi)}$, by Lemma~\ref{alg-xing}.
\end{proof}

\begin{lemma}
\label{xingtailswap} Let $\pi:I\rightarrow J$ be a bijection such
that $\pi(i)=i$ for all $i\in I\cap J$.  For $k,l\in I\setminus J$
with $k<l$, let $s_{\pi(k),\pi(l)}$ denote the transposition of
the boundary vertices $b_{\pi(k)}$ and $b_{\pi(l)}$, and let
$\pi^*=s_{\pi(k),\pi(l)}\circ\pi$. Then
\begin{equation*}(-1)^{\xing(\pi^*)}=
\begin{cases}{(-1)^{\xing(\pi)+1}} & \text{if $(k,l)$ is a crossing or an
alignment;}\\
{(-1)^{\xing(\pi)}} & \text{if $(k,l)$ is a
misalignment.}
\end{cases}
\end{equation*}
\end{lemma}

\begin{proof}
Applying Lemma~\ref{alg-xing} and simplifying, we obtain:
\begin{eqnarray*}
(-1)^{\xing(\pi)}(-1)^{\xing(\pi^*)} &=&
\operatorname{sgn}\left[\prod_{i_1<i_2}(i_1-\pi(i_2))(\pi(i_2)-\pi(i_1))(\pi(i_1)-i_2)\right]
\\ & & \cdot \operatorname{sgn}\left[\prod_{i_1<i_2}(i_1-\pi^*(i_2))(\pi^*(i_2)-\pi^*(i_1))(\pi^*(i_1)-i_2)\right]\\
&=&
\operatorname{sgn}\left[(k-\pi(l))(\pi(l)-l)(l-\pi(k))(\pi(k)-k)\right],
\end{eqnarray*}
and the lemma follows from Corollary~\ref{misalign}.
\end{proof}

\section{Proof of the main theorem in the perfectly oriented case}
\label{sec:involution}

Throughout this section, we assume that~$G$ is a perfectly
oriented graph.  Recall that $I=\{i_1<\cdots<i_k\}$ indexes the
set of boundary sources of $G$.

\begin{defn}\label{def:flow}
A subset $F$ of (distinct) edges in a perfectly oriented planar
circular directed graph $G$ is called a \emph{flow} if, for each
interior vertex $v$ in~$G$, the number of edges of $F$ that arrive
at~$v$ is equal to the number of edges of~$F$ that leave from~$v$.

A flow~$C$ is \emph{conservative} if it contains no edges incident
to the boundary. We denote by~$\mathcal{C}(G)$ the set of all
conservative flows in~$G$.

Let $J$ be a $k$-element subset of $[n]$.  We say that a flow $F$
 is a \emph{flow from $I$ to~$J$} if each boundary source~$b_i$ is
 connected by a walk in $F$ to a
(necessarily unique) boundary vertex~$b_j$ with $j\in J$. If $G$
is perfectly oriented, we denote by $\mathcal{F}_{J}(G)$ the set
of all flows from $I$ to~$J$.

The \emph{weight} of a flow~$F$, denoted~$\wt(F)$, is by
definition the product of the weights of all edges in~$F$. A flow
with no edges has weight 1.
\end{defn}

We note that each flow is a union of $k$ non-intersecting
self-avoiding walks, each connecting a boundary source~$b_i$
($i\in I$) to a distinct boundary vertex $b_j$ ($j\in J$),
together with a (possibly empty) collection of pairwise disjoint
cycles, none of which intersect any of the walks.  Further, each
flow lies in precisely one of the sets $\mathcal{F}_J(G)$.  In
particular, for a conservative flow, each of the $k$ walks between
boundary vertices is trivial, and
$\mathcal{C}(G)=\mathcal{F}_I(G)$.

Using the above definitions, we can restate
Theorem~\ref{th:main-intro} as follows.

\begin{thm}
\label{th:mainthm} Let $N=(G,x)$ be a perfectly oriented planar
circular network. Then the maximal minors of the boundary
measurement matrix $A(N)$ are given by
\begin{equation}
\label{eq:main-formula}
\Delta_J(A(N))=\frac{\displaystyle\sum_{F\in\mathcal{F}_J(G)}
\wt(F)}{\displaystyle\sum_{C\in\mathcal{C}(G)} \wt(C)}.
\end{equation}
\end{thm}

\begin{example}
Consider the planar circular network $N$ in
Figure~\ref{fig:mainex}, with $I=\{1,4\}$.  For $J=\{1,5\}$, let
us describe the set of flows $\mathcal{F}_{\{1,5\}}(G)$. The
boundary vertex $b_1$ must be connected to itself by the trivial
walk $b_1\leadsto b_1$. Together with the unique self-avoiding
walk $P:b_4\leadsto b_5$ of weight $a_4w_2fy_2a_5$, this gives a
flow from $\{1,4\}$ to $\{1,5\}$. There is one additional flow,
consisting of $P$ and the cycle of weight $Z=z_1z_2z_3z_4$ (along
with the trivial walk $b_1\leadsto b_1$). Thus, the numerator
of~\eqref{eq:main-formula} is
\[
\sum_{F\in\mathcal{F}_{\{1,5\}}(G)}\wt(F) =a_4w_2fy_2a_5(1+Z).
\]
The only cycles in the network $N$ are those of weights $W$, $Y$,
$Z$, and $T$.  Since conservative flows in $G$ are unions of
disjoint cycles, we have
\begin{eqnarray*}
\sum_{C\in\mathcal{C}(G)} \wt(C) &=& 1+W+Y+Z+T+WZ+WY+YZ+WYZ+ZT\\
&=&(1+Z)[(1+W)(1+Y)+T].
\end{eqnarray*}
Consequently,
\[
\Delta_{\{1,5\}}(A(N))=\frac{a_4w_2fy_2a_5(1+Z)}{(1+Z)[(1+W)(1+Y)+T]}=\frac{a_4w_2fy_2a_5}{(1+W)(1+Y)+T}.
\]
\end{example}

\begin{proof}[Proof of Theorem \ref{th:mainthm}]
For a bijection $\pi:I\rightarrow J$, let $\mathcal{P}_\pi$ denote
the set of all (possibly intersecting) collections of walks
$\mathbf{P}\!=\!(P_i)_{i\in I}$ connecting $I$ and $J$ in accordance
with~$\pi$:
\[
\mathcal{P}_\pi=\{\mathbf{P}=(P_i:b_{i}\leadsto b_{\pi(i)})_{i\in I}\}.
\]
In view of \eqref{eq:Mij} and~\eqref{eq:delta(A(N))},
we can rewrite the claim \eqref{eq:main-formula} as
\begin{equation}
\label{eq:main-claim} \sum_{C\in\mathcal{C}(G)} \
\sum_{\pi:I\rightarrow J} \ \sum_{\P\in \mathcal{P}_\pi} \wt(C,\P)
= \sum_{F\in\mathcal{F}_J(G)}\wt(F),
\end{equation}
where $\wt(C,\P)$, for $\P\in \mathcal{P}_\pi$, is defined by
\[
\wt(C,\P)=\wt(C)(-1)^{\xing(\pi)}\prod_{i\in
I}(-1)^{\wind(P_i)}\wt(P_i).
\]
Note that if $C$ and $\P$ form a flow~$F$ from $I$ to~$J$, then
$\xing(\pi)=0$ and $\wind(P_i)=0$ for all $i$, so that
$\wt(C,\P)=\wt(F)$. Hence \eqref{eq:main-claim} can be restated as
saying that all terms on its left-hand side cancel except for the
ones for which $C$ and $\P$ form a flow from~$I$ to~$J$. It
remains to construct a sign-reversing involution proving this
claim. More precisely, we need an involution $\varphi$ on the set
of pairs $(C,\P)$ such that \begin{enumerate}[(i)] \item
$C\in\mathcal{C}(G)$ is a conservative flow, \item $\P$ is a
collection of $k=|I|$ walks connecting $I$ and~$J$, and \item $C$
and $\P$ do \emph{not} form a boundary flow.\end{enumerate}
Furthermore, $\varphi$ must satisfy
$\wt(\varphi(C,\P))=-\wt(C,\P)$.


For a pair $(C,\P)$ satisfying (i)-(iii), we define
$\varphi(C,\P)=(C^*,\P^*)$ as follows. Let $\P=(P_i)_{i\in
I}\in\mathcal{P}_\pi$, with $\pi:I\to J$ a bijection. Choose the
smallest $i\in I$ such that $P_i$ is not self-avoiding or has a
common vertex with $C$ or with some $P_{i'}$ with $i'>i$. (Such an
$i$ exists by the assumptions we made regarding~$(C,\P)$.)

Let $P_i=(e_1,\ldots, e_m)$.  Choose the smallest $q$ such that
the edge $e_q$ lies in $C$ or in some $P_{i'}$ with $i'>i$, or
$e_q=e_r$ for some $r>q$.

\begin{itemize}
\item If $e_q$ lies in some $P_{i'}$ with $i'>i$, choose the
smallest such $i'$. (This case allows for the possibility that
$P_i$ intersects itself or $C$ at $e_q$.) We will then swap the
tails of $P_i$ and $P_{i'}$ as follows. Let $P_{i'}=(h_1, \ldots,
h_{m'})$, and choose the smallest $q'$ such that $h_{q'}=e_q$. Set
$P_i^*=(e_1, \ldots, e_{q-1}, e_q=h_{q'}, h_{q'+1}, \ldots,
h_{m'})$ and $P_{i'}^*=(h_1, \ldots, h_{q'-1}, h_{q'}=e_q,
e_{q+1}, \ldots, e_m)$. Set
$\P^*=\P\setminus\{P_i,P_{i'}\}\cup\{P_i^*,P_{i'}^*\}$ and set
$C^*=C$.
(Note that $q<\min(m,m')$ in this
case, so $\P^*\neq \P$.)

\item Otherwise we will find the first point along $P_i$ where we
can move a cycle from $C$ to $P_i$ or vice versa, as follows.  If
$P_i$ is not self-avoiding, let $\ell$ be the first cycle that
$P_i$ completes.  That is, choose the smallest $s$ such that
$e_r=e_s$ for some $r<s$; then $\ell=(e_r, e_{r+1}, \ldots,
e_{s-1})$.  If $P_i$ is self-avoiding, then set $s=\infty$.  If
$C$ intersects $P_i$, choose the smallest $t$ such that $e_t$
occurs in a (necessarily unique) cycle $L=(l_1,l_2 \ldots, l_w)$
in $C$, where $l_1=e_t$.  If $C\cap P_i=\emptyset$, then set
$t=\infty$. Note that at least one of $t$ or $s$ must be finite,
and $t\neq s$, since $e_s=e_r$ and $r<s$.

\begin{itemize}
\item If $t<s$, we move $L$ from $C$ to $P_i$, as follows. Set
$C^*=C\setminus\{L\}$, $P_i^*=(e_1, \ldots, e_{t-1},e_t=l_1,
\ldots, l_w, e_t, \ldots, e_m)$, and $\P^*=\P\setminus
 \{P_i\}\cup \{P_i^*\}$.

\item If $t>s$, we move $\ell$ from $P_i$ to $C$, as follows. Set
$C^*=C\cup\{\ell\}$, $P_i^*=(e_1, \ldots, e_{r-1},e_s,\ldots,
e_m)$, and $\P^*=\P\setminus\{P_i\}\cup\{P_i^*\}$.

\end{itemize}
\end{itemize}

It is easy to see that, with this definition, the image
$(C^*,\P^*)$ is again a pair of the required kind, i.e., it
satisfies the conditions (i)-(iii) above.\\

Let us verify that $\varphi$ is an involution.  First, we check
that $\varphi$ does not change the value of $i$.  That is, among
all walks in $\P^*$ which intersect themselves, another walk, or a
cycle in $C^*$, the walk with the smallest index (of its starting
point) is $P_i^*$.  Indeed, our moves only affect $P_i$, $P_{i'}$,
and $C$, keeping their combined set of edges intact, so the
involution will not introduce a new self-intersection in any $P_a$
with $a<i$, nor will it introduce an intersection between $P_a$
and any path or cycle.

Consider $\varphi(C,\P)=(C,\P^*)$ in the first case.  After
swapping tails, $P_i^*$ still has no intersections with $C$ or any
of the other paths before the edge $e_q$. Further, $P_i^*$ does
not have any self-intersections before $e_q$ (though it may have
self-intersections at $e_q$), since $P_i$ did not have any
self-intersections before $e_q$ and the tail of $P_{i'}$ did not
intersect $P_i$ before $e_q$. Thus, $e_q$ remains the first edge
along $P_i^*$ with an intersection. Now, $P_i^*$ and $P_{i'}^*$
intersect at this edge, and no path with smaller index intersects
$P_i^*$ at $e_q$, so we will swap the same tails again.

Consider the second case, with $\varphi(C,\P)=(C\setminus\{L\},
\P^*)$ or $\varphi(C,\P)=(C\cup\{\ell\}, \P^*)$.  Here, $P_i$
intersects itself or $C$ at $e_q$, but does not intersect any
other path at $e_q$. After moving a cycle, the same is true for
$P_i^*$.  (If the cycle moved starts at $e_q$, then either a
self-intersection becomes an intersection with $C$, or an
intersection with $C$ becomes a self-intersection. If the cycle
moved starts later, then the intersections at $e_q$ remain as they
are.)  If $P_i$ intersects $C$ before completing its first cycle,
then $P_i^*$ will complete its first cycle before intersecting
$C\setminus\{L\}$.  If $P_i$ completes its first cycle $\ell$
before intersecting $C$, then $P_i^*$ will intersect
$C\cup\{\ell\}$ before completing its first cycle.  Thus, the same
cycle is moved both times.  We have now shown that $\varphi$ is an
involution.\\


Finally, we verify that $\varphi$ is sign-reversing.  In the case
of tail swapping, we need to show that
$\wind(P_i)+\wind(P_{i'})+\xing(\pi)+\wind(P_i^*)+\wind(P_{i'}^*)+\xing(\pi^*)$
is odd, where $\pi^*$ is the bijection such that $\P^*\in
\mathcal{P}_{\pi^*}$. By Lemma~\ref{xingtailswap},
$\xing(\pi)+\xing(\pi^*)$ is even if and only if $(i,i')$ is a
misalignment. Thus, we need to show that $(i,i')$ is a
misalignment if and only if
\begin{equation}\label{windmod2}
\wind(P_i)+\wind(P_{i'})+\wind(P_i^*)+\wind(P_{i'}^*)\equiv 1
(\text{mod 2}).
\end{equation}

\psset{unit=1pt}
\begin{figure}[h]
\begin{pspicture}(-41,-41)(41,41)
\pscircle[linecolor=black](0,0){40}

\psline[linewidth=1pt](-34.64,20)(-15,0.75)
\psline[linewidth=1pt,arrows=->,arrowsize={1.5pt
4}](15,0.75)(34.64,20) \psline[linewidth=1pt](-15.5,1)(15,1)

\psline[linestyle=dashed,linewidth=1pt](-34.64,-20)(-15,-2)
\psline[linestyle=dashed,linewidth=1pt,arrows=->,arrowsize={1.5pt
4}](15,-2)(34.64,-20)
\psline[linestyle=dashed,linewidth=1pt](-15,-2)(15,-2)

\uput[r](41,0){$\leftrightsquigarrow$}

\end{pspicture}
\qquad
\begin{pspicture}(-41,-41)(41,41)
\pscircle[linecolor=black](0,0){40}

\psline[linewidth=1pt](-34.64,20)(-15,0.75)
\psline[linewidth=1pt,arrows=->,arrowsize={1.5pt
4}](15,0.75)(34.64,-20) \psline[linewidth=1pt](-15.5,1)(15,1)

\psline[linestyle=dashed,linewidth=1pt](-34.64,-20)(-15,-2)
\psline[linestyle=dashed,linewidth=1pt,arrows=->,arrowsize={1.5pt
4}](15,-2)(34.64,20)
\psline[linestyle=dashed,linewidth=1pt](-15,-2)(15,-2)

\end{pspicture}
\qquad
\begin{pspicture}(-41,-41)(41,41)
\pscircle[linecolor=black](0,0){40}

\psline[linewidth=1pt](-34.64,20)(-15,0.75)
\psline[linewidth=1pt,arrows=->,arrowsize={1.5pt
4}](15,0.75)(34.64,20) \psline[linewidth=1pt](-15.5,1)(15,1)

\psline[linestyle=dashed,linewidth=1pt,arrows=->,arrowsize={1.5pt
4}](15,-2)(-20,-34.64)
\psline[linestyle=dashed,linewidth=1pt](-15,-2)(20,-34.64)
\psline[linestyle=dashed,linewidth=1pt](-15,-2)(15,-2)

\uput[r](41,0){$\leftrightsquigarrow$}

\end{pspicture}
\qquad
\begin{pspicture}(-41,-41)(41,41)
\pscircle[linecolor=black](0,0){40}

\psline[linewidth=1pt](-34.64,20)(-15,0.75)
\psline[linewidth=1pt,arrows=->,arrowsize={1.5pt
4}](15,0.75)(-20,-34.64) \psline[linewidth=1pt](-15.5,1)(15,1)

\psline[linestyle=dashed,linewidth=1pt](20,-34.64)(-15,-2)
\psline[linestyle=dashed,linewidth=1pt,arrows=->,arrowsize={1.5pt
4}](15,-2)(34.64,20)
\psline[linestyle=dashed,linewidth=1pt](-15,-2)(15,-2)

\end{pspicture}
\qquad
\caption{Winding index and tail swapping}
\label{fig:wind-tail-swap}
\end{figure}
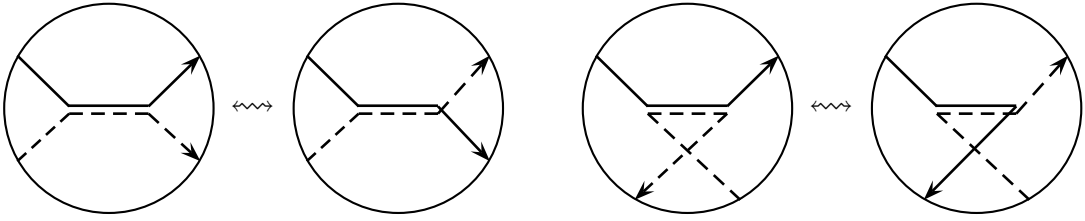
\psset{unit=1pt}

This statement is in fact true for \emph{any} instance of tail
swapping, i.e., it does not rely on our particular choice of the
walks $P_i$ and $P_{i'}$ sharing an edge $e_q$.  Viewing
(\ref{windmod2}) as a purely topological condition, we can
``unwind" each of the 4 subwalks from which our walks $P_i$ and
$P_{i'}$ are built, keeping $e_q$ fixed.  This will not change the
parity in (\ref{windmod2}) since each loop contained entirely in
one of the initial or terminal subwalks will contribute twice,
once to $\wind(P_i)+\wind(P_{i'})$ and once to
$\wind(P_i^*)+\wind(P_{i'}^*)$.  Deforming the walks as necessary,
we then obtain one of the four pictures shown in
Figure~\ref{fig:wind-tail-swap}. The last two of the four pictures
represent misalignments, and indeed, these are precisely the two
cases in which (\ref{windmod2}) holds.

In the remaining case (moving a cycle from $C$ to $\P$ or vice
versa), $\wind(P_i)$ changes parity, while $\xing(\pi)$ and all
other winding numbers do not change. Hence
$\wt(\varphi((C,\P))=-\wt(C,\P)$, as desired.
\end{proof}

\section{Extending to planar graphs with arbitrary
orientations} \label{sec:general} In this section, we provide an
extension of Theorem~\ref{th:mainthm} for arbitrarily oriented
planar networks. The proof relies on Postnikov's process in
\cite{post} for transforming an arbitrary planar circular network
into a partially specialized perfectly oriented planar circular
network.

When $G$ is a perfectly oriented graph, the following definition
is equivalent to Definition~\ref{def:flow}.  This extension
provides the appropriate setup for working in arbitrarily oriented
graphs~$G$.

\begin{defn}\label{def:altflow} A subset $F$ of (distinct)
edges in a planar circular directed graph $G$ (not necessarily
perfectly oriented) is called an \emph{alternating flow} if,
 for each interior vertex $v$ in~$G$, the edges $e_1, \ldots, e_d$ of $F$ which are incident to
$v$, listed in clockwise order around $v$, alternate in
orientation (that is, directed towards $v$ or directed away from
$v$).

In an alternating flow $F$, define the walks $W_i$ (with $i\in I$)
as follows. If $b_i$ is isolated in $F$, set $W_i$ to be the
trivial walk from $b_i$ to itself. Otherwise, let $W_i$ be the
unique path leaving $b_i$ which, at each subsequent vertex, takes
 the first left turn in~$F$, until it arrives at another boundary vertex.

For a $k$-element subset $J$ of $[n]$, we say that an alternating
flow $F$ is a \emph{flow from $I$ to~$J$} if each boundary
source~$b_i$ is connected by $W_i$ to a boundary vertex~$b_j$ with
$j\in J$. (The vertices $b_j$ are necessarily distinct.) Let
$\mathcal{A}_J$ denote the set of alternating flows from $I$ to
$J$.  In particular, $\mathcal{A}_I$ is precisely the set of
conservative alternating flows.
\end{defn}

\begin{defn}\label{def:power2}
Suppose $F$ is an alternating flow.  For each vertex $v$ in $G$,
let $\tau(v,F)$ denote the number of edges of $F$ coming into $v$.
Set
\[
\theta(F)=\sum_v\operatorname{max}\{\tau(v,F)-1,0\}.
\]
\end{defn}

\begin{cor}\label{th:main-non-perfect}
Let $N=(G,x)$ be a planar circular network with source set indexed
by $I$. Then the maximal minors of the boundary measurement matrix
$A(N)$ are given by the formula
\begin{equation}
\label{eq:general-formula}
\Delta_J(A(N))=\frac{\displaystyle\sum_{F\in\mathcal{A}_J(G)}
2^{\theta(F)}\wt(F)}{\displaystyle\sum_{C\in\mathcal{A}_I(G)}
2^{\theta(C)}\wt(C)}.
\end{equation}
\end{cor}

The proof of Corollary~\ref{th:main-non-perfect} will follow
Proposition~\ref{prop:maketrivalent}, which describes Postnikov's
transformation process in detail.\\

Let $E(G)$ denote the edge set of $G$.  Let
$\Delta_J(A(N))(\alpha)$ denote the evaluation of the
subtraction-free rational expression $\Delta_J(A(N))$ under a
specialization map $\alpha:E(G)\rightarrow \mathbb{R}$ assigning a
positive real weight $\alpha_e$ to each edge $e$.

For boundary measurements, there is no loss of generality in
assuming that $G$ has no internal sources or sinks.  Further, we
may assume that there are no vertices of degree 2. Indeed, if
there is a vertex $v$ with exactly one incoming edge $e_1$ and
exactly one outgoing edge $e_2$, we may remove $v$ and glue $e_1$
and $e_2$ into a single edge $e$ of weight $x_e=x_{e_1}x_{e_2}$.

\begin{prop}[\cite{post}] \label{prop:maketrivalent}Let $N=(G,x)$ be a planar network
with boundary sources indexed by $I$ and positive weight function
$\alpha:E(G)\rightarrow \mathbb{R}$, with $x_e\mapsto \alpha_e$.
Let $N'=(G',x')$ and $\alpha':E(G')\rightarrow\mathbb{R}$ (with
$x'_e\mapsto \alpha'_e$) denote a perfectly oriented planar
network and corresponding positive weight function obtained from
$N$ and $\alpha$ by the process described below. Then for all
$J\subset [n]$ with $|J|=|I|$, we have
\[\Delta_J(A(N))(\alpha)=\Delta_J(A(N'))(\alpha') .\]

To obtain $N'$ from $N$, we perform the following operations in
stages (1)-(3).
\begin{enumerate}
\item First, suppose that $N$ has an internal vertex $v$ of degree
greater than 3; let $e_1, \ldots, e_d$ denote the edges incident
to $v$, listed in clockwise order. If two adjacent edges $e_i$ and
$e_{i+1}$ (modulo $n$) have the same orientation, either both
towards $v$ or both away from $v$, we choose such a pair, pull
these edges away from $v$, insert a new vertex $v'$ and a new edge
$e$ (directed from $v'$ to $v$ when $e_i$ and $e_{i+1}$ are edges
entering $v$ and from $v$ to $v'$ when $e_i$ and $e_{i+1}$ are
edges leaving $v$), and attach the edges $e_i$ and $e_{i+1}$ to
$v'$. (See Figure~\ref{fig:pulloutedges}.) We set $\alpha'_{e}=1$.
Repeat until the resulting network has no vertices $v$ of this
form.

\psset{unit=0.9pt}
\begin{center}
\begin{figure}[ht]
\begin{pspicture}(0,0)(380,80)

\psline[linewidth=1pt,arrows=->,arrowsize={1.5pt 3}](16,64)(40,40)
\psline[linewidth=1pt,arrows=->,arrowsize={1.5pt 3}](16,16)(40,40)
\psline[linewidth=1pt,arrows=->,arrowsize={1.5pt 3}](40,40)(64,16)
\psline[linewidth=1pt,arrows=->,arrowsize={1.5pt 3}](40,40)(64,64)
\psline[linewidth=1pt,arrows=->,arrowsize={1.5pt 3}](40,40)(40,68)
\uput[l](16,64){$\alpha_{e_5}$} \uput[l](16,16){$\alpha_{e_4}$}
\uput[r](64,16){$\alpha_{e_3}$} \uput[r](64,64){$\alpha_{e_2}$}
\uput[u](40,68){$\alpha_{e_1}$}

\uput[r](80,40){$\leadsto$}

\psline[linewidth=1pt,arrows=->,arrowsize={1.5pt
3}](126,64)(150,40)
\psline[linewidth=1pt,arrows=->,arrowsize={1.5pt
3}](126,16)(150,40)
\psline[linewidth=1pt,arrows=->,arrowsize={1.5pt
3}](150,40)(180,40)
\psline[linewidth=1pt,arrows=->,arrowsize={1.5pt
3}](180,40)(204,16)
\psline[linewidth=1pt,arrows=->,arrowsize={1.5pt
3}](180,40)(204,64)
\psline[linewidth=1pt,arrows=->,arrowsize={1.5pt
3}](180,40)(180,68) \uput[l](126,64){$\alpha_{e_5}$}
\uput[l](126,16){$\alpha_{e_4}$} \uput[r](204,16){$\alpha_{e_3}$}
\uput[r](204,64){$\alpha_{e_2}$}
\uput[u](180,68){$\alpha_{e_1}$}\uput[d](165,40){$1$}

\uput[r](220,40){$\leadsto$}

\psline[linewidth=1pt,arrows=->,arrowsize={1.5pt
3}](266,64)(290,40)
\psline[linewidth=1pt,arrows=->,arrowsize={1.5pt
3}](266,16)(290,40)
\psline[linewidth=1pt,arrows=->,arrowsize={1.5pt
3}](290,40)(320,40)
\psline[linewidth=1pt,arrows=->,arrowsize={1.5pt
3}](320,40)(350,40)
\psline[linewidth=1pt,arrows=->,arrowsize={1.5pt
3}](350,40)(374,16)
\psline[linewidth=1pt,arrows=->,arrowsize={1.5pt
3}](350,40)(374,64)
\psline[linewidth=1pt,arrows=->,arrowsize={1.5pt
3}](320,40)(320,68) \uput[l](266,64){$\alpha_{e_5}$}
\uput[l](266,16){$\alpha_{e_4}$} \uput[r](374,16){$\alpha_{e_3}$}
\uput[r](374,64){$\alpha_{e_2}$}
\uput[u](320,68){$\alpha_{e_1}$}\uput[d](305,40){$1$}\uput[d](335,40){$1$}

\end{pspicture}
\qquad \caption{Pulling out adjacent edges with the same
orientation.} \label{fig:pulloutedges}
\end{figure}
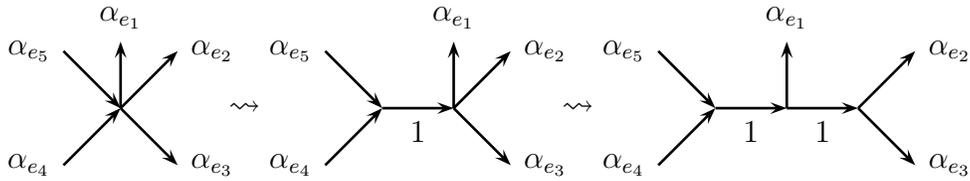
\end{center}
\psset{unit=1pt}

\item  If a vertex $v$ of degree greater than 3 remains, its
incident edges must alternate orientation in clockwise order. In
this case we blow up the vertex $v$ into a cycle with edges all
oriented clockwise, as in Figure~\ref{fig:blowupvertex}.  If $e$
is an edge coming into $v$, we set $\alpha'_e=2\alpha_e$, and if
$e$ is one of the new edges created to make the cycle, we set
$\alpha'_e=1$.  Repeat until the resulting network has no vertices
$v$ of this form.

\psset{unit=0.9pt}
\begin{center}
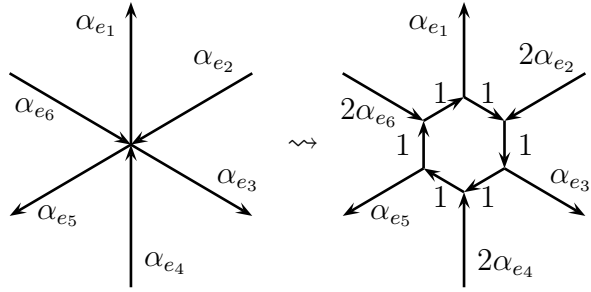
\begin{figure}[ht]
\begin{pspicture}(0,0)(280,120)

\psline[linewidth=1pt,arrows=->,arrowsize={1.5pt 3}](9,90)(60,60)
\psline[linewidth=1pt,arrows=->,arrowsize={1.5pt 3}](60,60)(9,30)
\psline[linewidth=1pt,arrows=->,arrowsize={1.5pt 3}](60,0)(60,60)
\psline[linewidth=1pt,arrows=->,arrowsize={1.5pt
3}](60,60)(111,30)
\psline[linewidth=1pt,arrows=->,arrowsize={1.5pt
3}](111,90)(60,60)\psline[linewidth=1pt,arrows=->,arrowsize={1.5pt
3}](60,60)(60,120)

\uput[d](20,85){$\alpha_{e_6}$} \uput[r](15,30){$\alpha_{e_5}$}
\uput[r](60,10){$\alpha_{e_4}$} \uput[u](105,35){$\alpha_{e_3}$}
\uput[u](95,85){$\alpha_{e_2}$} \uput[l](60,110){$\alpha_{e_1}$}

\uput[r](120,60){$\leadsto$}

\psline[linewidth=1pt,arrows=->,arrowsize={1.5pt
3}](149,90)(183,70)
\psline[linewidth=1pt,arrows=->,arrowsize={1.5pt
3}](183,50)(149,30)
\psline[linewidth=1pt,arrows=->,arrowsize={1.5pt
3}](200,0)(200,40)
\psline[linewidth=1pt,arrows=->,arrowsize={1.5pt
3}](217,50)(251,30)
\psline[linewidth=1pt,arrows=->,arrowsize={1.5pt
3}](251,90)(217,70)\psline[linewidth=1pt,arrows=->,arrowsize={1.5pt
3}](200,80)(200,120)

\psline[linewidth=1pt,arrows=->,arrowsize={1.5pt
3}](183,50)(183,70)\psline[linewidth=1pt,arrows=->,arrowsize={1.5pt
3}](200,40)(183,50)\psline[linewidth=1pt,arrows=->,arrowsize={1.5pt
3}](217,50)(200,40)\psline[linewidth=1pt,arrows=->,arrowsize={1.5pt
3}](217,70)(217,50)\psline[linewidth=1pt,arrows=->,arrowsize={1.5pt
3}](200,80)(217,70)\psline[linewidth=1pt,arrows=->,arrowsize={1.5pt
3}](183,70)(200,80)

\uput[l](183,60){$1$}\uput[r](217,60){$1$}\uput[u](190,72){$1$}
\uput[u](210,72){$1$}\uput[d](190,48){$1$}\uput[d](210,48){$1$}

\uput[d](160,85){$2\alpha_{e_6}$} \uput[r](155,30){$\alpha_{e_5}$}
\uput[r](200,10){$2\alpha_{e_4}$} \uput[u](245,35){$\alpha_{e_3}$}
\uput[u](235,85){$2\alpha_{e_2}$}
\uput[l](200,110){$\alpha_{e_1}$}

\end{pspicture}
\qquad \caption{Blowing up a vertex with alternating edge
directions.} \label{fig:blowupvertex}
\end{figure}
\end{center}
\psset{unit=1pt}

\item Finally, for any remaining edge $e$ unaffected by these
steps (i.e. such that $\alpha'_e$ has not yet been specified), set
$\alpha'_e=\alpha_e$.  Let $N'$ and $\alpha'$ denote the final
result.
\end{enumerate}
\end{prop}

By \emph{contracting} an edge $e$, we mean removing the edge $e$
and identifying its two endpoints. (If we contract all edges in a
connected subset of edges, the image is a single vertex.) It is
easy to see that by contracting all new edges created in
Proposition~\ref{prop:maketrivalent} above, we obtain $G$ from
$G'$.

\begin{defn}\label{def:blowupvertices}Let $B(G)$ denote the set of vertices of $G$ around which the
orientations of edges switch at least four times. Call such a
vertex $v$ a \emph{blowup vertex}. These are precisely the
vertices which are blown up into cycles in the second stage of
Proposition~\ref{prop:maketrivalent}.

For an alternating flow $F$ in a planar network $N$, we define
$\epsilon(F)$ to be the number of edges of $F$ which enter a
blowup vertex of $G$, $\beta(F)$ to be the number of blowup
vertices of $G$ which occur as the endpoint of some edge in $F$,
and $\eta(F)$ to be the number of blowup vertices of $G$ which are
not endpoints of any edge in $F$.  Recalling
Definition~\ref{def:power2}, note that
$\theta(F)=\epsilon(F)-\beta(F)$.
\end{defn}

\begin{proof}[Proof of Corollary~\ref{th:main-non-perfect}]
Fix an image $N'$ of $N$ under the transformation in
Proposition~\ref{prop:maketrivalent}, and let $F'$ be an
alternating flow in $N'$. It is easily verified that contracting
all edges in $E(G')- E(G)$ gives a bijection between alternating
flows $F'$ in $G'$ and pairs $(F,A)$, where $F$ is an alternating
flow in $G$ and $A$ is a subset of vertices in $B(G)$ which are
not endpoints of any edges in $F$. Further, extending $\alpha$ and
$\alpha'$ linearly, we have
\[\alpha'(\wt(F'))=2^{\epsilon(F)}\alpha(\wt(F)),\]
and there are $2^{\eta(F)}$ flows $F'$ corresponding to a given
flow $F$, all with the same weight.

Since this relationship holds for every positive specialization
$\alpha$, Theorem~\ref{th:mainthm} and
Proposition~\ref{prop:maketrivalent} then imply that

\begin{equation*}
\Delta_J(A(N))=\frac{\displaystyle\sum_{F\in\mathcal{A}_J(G)}
2^{\epsilon(F)+\eta(F)}\wt(F)}{\displaystyle\sum_{C\in\mathcal{A}_I(G)}
2^{\epsilon(C)+\eta(C)}\wt(C)} .
\end{equation*}

Cancelling a factor of $|B(G)|=\eta(F)+\beta(F)$ from each term in
the numerator and denominator, we obtain the formula
(\ref{eq:general-formula}), as desired.
\end{proof}

\section{Notes on Pl\"ucker coordinates for perfectly oriented non-planar networks}
\label{sec:non-planar}

It is natural to ask to what extent we can develop these
constructions in the non-planar setting.  While the notion of the
topological winding index only makes sense for planar graphs,
Lawler's notion of loop-erasure in \cite{lawpaper} allows us to
give a non-planar analogue of the winding index if $G$ is
perfectly oriented.  In this non-planar setting, we no longer have
the positivity results, but we can describe those Pl\"ucker
coordinates which are individual boundary measurements.

We begin by extending the definition of circular directed graphs
and networks (Definition~\ref{def:circ}) to suit the non-planar
setting. For a general \emph{circular directed graph}, we no
longer require that $G$ has a planar embedding in a disk, but we
still ask for the boundary vertices to be labeled in cyclic order
and for each boundary vertex to be adjacent to at most one edge.

\begin{defn}[\cite{lew,lawbook}]
\label{looperasure} The \emph{loop-erased part} of a walk
$P:b_i\leadsto b_j$, denoted ${\rm LE}(P)$, is defined recursively
as follows. If $P=(e_1, \ldots, e_m)$ does not have any
self-intersections, then ${\rm LE}(P)=P$. Otherwise, we set ${\rm
LE}(P)={\rm LE}(P')$, where $P'$ is obtained from $P$ by removing
the first cycle it completes. More precisely, when $G$ is
perfectly oriented, find the smallest value of $s$ such that there
exists $r<s$ with $e_r=e_s$, and remove the segment $e_r, e_{r+1},
\ldots, e_{s-1}$ from~$P$ to obtain~$P'$. The \emph{loop-erasure
number} $\loop(P)$ is defined as the number of cycles erased
during the calculation of ${\rm LE}(P)$. With the notation as
above, we have $\loop(P)=\loop(P')+1$, and $\loop(P)=0$ when $P$
is a self-avoiding walk.
\end{defn}

\begin{prop}[\cite{post}]
\label{loop=wind} Suppose that $G$ is a perfectly oriented planar
circular directed graph.  If $P$ is a walk from a boundary vertex
$b_i$ to a boundary vertex $b_j$, then
$(-1)^{\loop(P)}=(-1)^{\wind(P)}$.
\end{prop}

\begin{proof}
Each boundary vertex is incident to at most one edge, so $P$ has
no self-intersections at its endpoints. Since $G$ is perfectly
oriented, $P$ repeats at least one edge at every
self-intersection. The claim then follows by induction on
$\loop(P)$, as an erasure of a cycle changes the winding index
by~$\pm 1$.
\end{proof}

Proposition~\ref{loop=wind} allows us to view $\loop(P)$ as a
natural generalization of $\wind(P)$ for the non-planar case. This
observation leads to an extension of Postnikov's construction
(which applies to planar networks and employs the winding index)
to arbitrary perfectly oriented graphs.
Definitions~\ref{def:perf},~\ref{def:wind},~\ref{def:bdry-meas},~\ref{def:boundarymeasmatrix},
~\ref{def:xing-number}, and~\ref{def:flow} then extend to
perfectly oriented non-planar networks in the obvious way,
replacing $\wind(P)$ with $\loop(P)$ wherever appropriate.

\begin{cor}\label{th:main-non-planar}  Suppose $N=(G,x)$ is a
perfectly oriented circular network with boundary source set
indexed by $I$.  Then, for $i\in I$ and $j\in [n]$, we have

\[
M_{ij}=\Delta_{(I\setminus\{i\})\cup\{j\}}(A(N))
=\frac{\displaystyle\sum_{F\in\mathcal{F}_{(I\setminus\{i\})\cup\{j\}}(G)}
\wt(F)}{\displaystyle\sum_{C\in\mathcal{C}(G)} \wt(C)}.
\]

\end{cor}

\begin{proof}
This follows directly from the proof of Theorem~\ref{th:mainthm},
since for these special Pl\"ucker coordinates, the tail swapping
process of the proof is never called upon.
\end{proof}

Although the result holds for those minors $\Delta_J$ which are
boundary measurements $M_{ij}$, it is generally not valid for the
remaining Pl\"ucker coordinates. For non-planar networks, tail
swapping does not always yield the sign change in $(-1)^{\xing}$
that we obtain in the planar case.  As a result, the numerator and
denominator of a minor $\Delta_J$ do not necessarily simplify to
linear combinations of flow weights.

\begin{example}
Consider the network $N$ in Figure~\ref{fig:non-planar}, with
boundary measurement matrix $A(N)$ below.  The minors
$\Delta_{12}$, $\Delta_{13}$, $\Delta_{14}$, $\Delta_{24}$, and
$\Delta_{34}$ all satisfy Theorem~\ref{th:mainthm}.
\[
A(N)=\left(%
\begin{array}{cccc}
  1 & \frac{a_1fa_2}{1+cdef} & \frac{a_1feda_3}{1+cdef} & 0 \\
  0 & \frac{a_4dcfa_2}{1+cdef} & \frac{a_4da_3}{1+cdef} & 1 \\
\end{array}%
\right)
\]

However, for $\Delta_{23}$, we do not get the desired
cancellation; in the simplified rational expression, both the
numerator and denominator are quadratic in flow weights. We have
\[
\Delta_{23}(A(N))= \frac{a_1a_2a_3a_4\cdot
df(1-cdef)}{(1+cdef)^2}~.
\]
\end{example}

\psset{unit=9pt}
\begin{center}
\begin{figure}[ht]
\begin{pspicture}(-2,2)(32,14)

\psline[linestyle=dotted](0,2)(0,14)(30,14)(30,2)(0,2)

\psline[linewidth=2pt,arrows=*->,arrowsize={1.5pt 4}](0,4)(10,12)
\psline[linewidth=2pt,arrows=*-,arrowsize={1.5pt
4}](10,4)(5.3,7.76)
\psline[linewidth=2pt,arrows=->,arrowsize={1.5pt
4}](4.7,8.24)(0,12)
\psline[linewidth=2pt,arrows=*-,arrowsize={1.5pt 4}](0,12)(0,12)
\psline[linewidth=2pt,arrows=*->,arrowsize={1.5pt 4}](10,12)(10,4)
\psline[linewidth=2pt,arrows=*->,arrowsize={1.5pt 4}](10,4)(20,4)
\psline[linewidth=2pt,arrows=*->,arrowsize={1.5pt 4}](20,4)(20,12)
\psline[linewidth=2pt,arrows=*->,arrowsize={1.5pt
4}](20,12)(10,12)
\psline[linewidth=2pt,arrows=*->,arrowsize={1.5pt 4}](30,4)(20,4)
\psline[linewidth=2pt,arrows=*->,arrowsize={1.5pt
4}](20,12)(30,12) \psline[linewidth=2pt,arrows=*-,arrowsize={1.5pt
4}](30,12)(30,12)

\uput[l](0,4){\large{$b_1$}} \uput[l](0,12){\large{$b_2$}}
\uput[r](30,12){\large{$b_3$}} \uput[r](30,4){\large{$b_4$}}

\uput[u](3.5,9.5){$a_2$} \uput[d](3.5,6.5){$a_1$}
\uput[u](15,12){$c$} \uput[r](10,8){$f$} \uput[r](20,8){$d$}
\uput[d](15,4){$e$} \uput[u](25,12){$a_3$} \uput[d](25,4){$a_4$}

\end{pspicture}
\qquad \caption{Boundary measurements in a non-planar network.}
\label{fig:non-planar}
\end{figure}
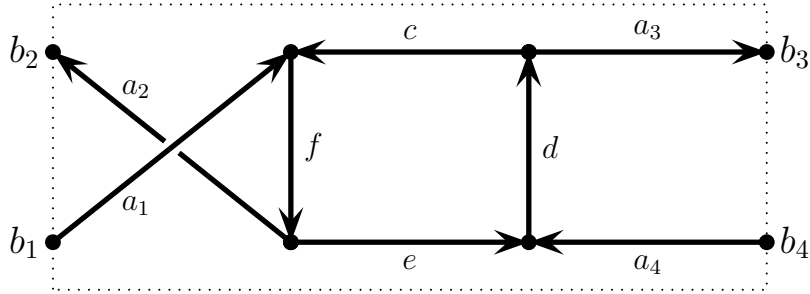
\end{center}
\psset{unit=1pt}

\begin{remark}
If we consider flow weights as polynomials with coefficients in
the finite field of two elements, $\mathbb{F}_2$, then
(\ref{eq:main-formula}) holds for all $\Delta_J$ in the perfectly
oriented non-planar case; this also follows directly from the
proof of Theorem~\ref{th:mainthm}.
\end{remark}

\section*{Acknowledgements}
The author would like to thank Alexander Postnikov for updates on
the work that inspired this paper; Alek Vainshtein, Lauren
Williams, and Gregg Musiker for insightful conversations; and
Sergey Fomin for many helpful comments on preliminary versions of
the manuscript.

\end{document}